\documentclass[oneside,english]{amsart}
\usepackage[T1]{fontenc}
\usepackage[latin9]{inputenc}
\usepackage{textcomp}
\usepackage{amstext}
\usepackage{amsthm}
\usepackage{amssymb}

\makeatletter
\numberwithin{equation}{section}
\numberwithin{figure}{section}
\numberwithin{table}{section}
\theoremstyle{plain}
\newtheorem{thm}{\protect\theoremname}[section]
\theoremstyle{definition}
\newtheorem{defn}[thm]{\protect\definitionname}
\theoremstyle{plain}
\newtheorem{lem}[thm]{\protect\lemmaname}

\makeatother

\usepackage{babel}
\providecommand{\definitionname}{Definition}
\providecommand{\lemmaname}{Lemma}
\providecommand{\theoremname}{Theorem}

\begin{document}
\title{Essential CR-sets near zero}
\author{Sujan Pal}
\address{Department of Mathematics, University of Kalyani, Kalyani, Nadia-741235,
West Bengal, India}
\email{\emph{sujan2016pal@gmail.com}}
\keywords{Stone-\v{C}ech compactification of discrete semigroup, Combinatorially
Rich sets, Ultrafilters near zero.}
\subjclass[2000]{Primary: 05D10; Secondary: 22A15.}
\begin{abstract}
Combinatorially Rich sets were introduced by Bergelson and Glasscock
for commutative semigroup. Latter Hindman, Hosseini, Strauss and Tootkaboni
extended the definition of Combinatorially Rich sets for arbitrary
semigroup. Recently Goswami proved that product of two Combinatorially
Rich sets is also a Combinatorially Rich set. On the other hand Hindman
and Leader were the first to introduce the concept of central sets
near zero for dense semigroups of $\left(\left(0,\infty\right),+\right)$
and demonstrated an important combinatorial consequence regarding
these sets. In this article we provided dynamical and combinatorial
characterization of essential CR-sets near zero and explore the cartesian
product of these sets.
\end{abstract}

\maketitle

\section{Introduction}

Hindman's theorem and Van der Waerden theorems are classical Ramsey
theoretic results, which deals with finite configurations and the
second one deals with infinite configurations.
\begin{thm}
\cite[Theorem 3.1]{key-11-1} (Hindman's theorem) Given any finite
coloring of the positive integers, there exists an infinite monochromatic
set $A$ such that the set $\left\{ \sum_{a\in F}a:F\in\mathcal{P}_{f}\left(A\right)\right\} $
is monochromatic.
\end{thm}

A set $A\subseteq\mathbb{N}$ is called IP-set if there exists a sequence
$\left\langle x_{n}\right\rangle _{n=1}^{\infty}$ such that $\text{FS}(x_{n})=\left\{ \sum_{t\in\alpha}x_{t}:\alpha\in\mathcal{P}_{f}(\mathbb{N})\right\} \subseteq A$.
(This definitions make perfect sense in any semigroup (S, ·) and we
use it in this context. FS is an abbriviation of finite sums and will
be replaced by FP if we use multiplicative notation for the semigroup
operation.) Consequently, Hindman's theorem say that if we finitely
color the positive integers, then one cell must be an IP-set\@. On
the other hand, Van der Waerden\textquoteright s Theorem \cite{key-16}
states that for any partition of the positive integers $\mathbb{N}$
one of the cells of the partition contains arbitrarily long arithmetic
progressions.

Using dynamical system Furstenberg gives a simulteneous extention
of both theorems, known as Central Sets Theorem \cite{key-10}. Before
further procceding, let us introduce some algebraic preliminaries
on Stone-\v{C}ech compactification.

There are many notions of largeness in a semigroup $\left(S,+\right)$
that are related to the algebraic structure of the Stone-\v{C}ech
compactification of the discrete set $S$. In our work we will use
the structure of Stone-\v{C}ech compactification of discrete semigroup.
Let $\left(S,\cdot\right)$ be a discrete semigroup and $\beta S$
be the Stone-\v{C}ech compactification of the discrete semigroup
$S$ and on $\beta S$ is the extension of $'\cdot'$ on $S$. The
points of $\beta S$ are ultrafilters and principal ultrafilters are
identified by the points of $S$. The extension is unique extension
for which $\left(\beta S,\cdot\right)$ is compact, right topological
semigroup with $S$ contained in its topological center. That is,
for all $p\in\beta S$ the function $\rho_{p}:\beta S\to\beta S$
is continuous, where $\rho_{p}(q)=q\cdot p$ and for all $x\in S$,
the function $\lambda_{x}:\beta S\to\beta S$ is continuous, where
$\lambda_{x}(q)=x\cdot q$. For $p,q\in\beta S$, $p\cdot q=\left\{ A\subseteq S:\left\{ x\in S:x^{-1}A\in q\right\} \in p\right\} $,
where $x^{-1}A=\left\{ y\in S:x\cdot y\in A\right\} $.

There is a famous theorem due to Ellis \cite[Corollary 2.39]{key-13}
that if $S$ is a compact right topological semigroup then the set
of idempotents $E\left(S\right)\neq\emptyset$. A non-empty subset
$I$ of a semigroup $T$ is called a left ideal of $S$ if $TI\subseteq I$,
a right ideal if $IT\subseteq I$, and a two sided ideal (or simply
an ideal) if it is both a left and right ideal. A minimal left ideal
is the left ideal that does not contain any proper left ideal. Similarly,
we can define minimal right ideal. Any compact Hausdorff right topological
semigroup $T$ has the smallest two sided ideal, $\begin{array}{ccc}
K(T) & = & \bigcup\left\{ L:L\text{ is a minimal left ideal of }T\right\} \\
 & = & \bigcup\left\{ R:R\text{ is a minimal right ideal of }T\right\} .
\end{array}$ 

Given a minimal left ideal $L$ and a minimal right ideal $R$, $L\cap R$
is a group, and in particular contains an idempotent. If $p$ and
$q$ are idempotents in $T$ we write $p\leq q$ if and only if $pq=qp=p$.
An idempotent is minimal with respect to this relation if and only
if it is a member of the smallest ideal $K(T)$ of $T$. See \cite{key-13}
for an elementary introduction to the algebra of $\beta S$ and for
any unfamiliar details. Central sets are the elements from the idempotent
in $K\left(\beta S\right)$.

In the present work we consider only those subsemigroups $S$ which
are dense in $\left(\left(0,\infty\right),+\right)$. In this case
one can define 
\[
0^{+}\left(S\right)=\left\{ p\in\beta S_{d}:\text{ for all }\left(\delta>0\right)\left(\left(0,\delta\right)\in p\right)\right\} 
\]
 ($S_{d}$ is the set $S$ with the discrete topology).

It was proved in \cite[Lemma 2.5]{key-12-1} that $0^{+}\left(S\right)$
is a compact right topological subsemigroup of $\left(\beta S,+\right)$.
It was also noted there in $0^{+}\left(S\right)$ is disjoint from
$K\left(\beta S\right)$ and hence gives some new information which
is not available from $K\left(\beta S\right)$. Being a compact right
topological semigroup , $0^{+}\left(S\right)$ has a minimal ideal
$K\left(0^{+}\left(S\right)\right)$. Central sets near zero are the
sets belong to the idempotent of $K\left(0^{+}\left(S\right)\right)$.

In this article we provide both dynamical and ellementary characterizations
of essential CR-set near zero in section 3. In section 4 we prove
that product of two CR-sets near zero is also a CR-set near zero.
We also prove the same result hold for essential CR-sets near zero.

\section{The notion of combinatorially rich sets near zero}

Bergelson and Glasscock introduced the notion of Combinatorially Rich
sets for commutative semigroups in \cite{key-16}. In \cite{key-12}
authors extended the notion of combinatorially rich sets to arbitrary
semigroups and investigated its relationship with other notion of
largeness.

We write $A^{B}$ for the set of functions from $B$ to $A$. In particular,
$S^{\mathbb{N}}$ is the set of sequences in $S$. And we write $\mathcal{P}_{f}\left(X\right)$
for the set of finite nonempty subsets of $X$, $\mid X\mid$ is denoted
the cardinality of the set $X$. The following is where we revisit
the concepts of J-sets and CR-sets.
\begin{defn}
\cite[Definition 2.1 and 2.2]{key-12} Let $\left(S,+\right)$ be
a commutative semigroup and let $A\subseteq S$. 
\end{defn}

\begin{enumerate}
\item $A$ is said to be a J-set if and only if whenever $F\in\mathcal{P}_{f}\left(S^{\mathbb{N}}\right)$,
there exist $a\in S$ and $H\in\mathcal{P}_{f}\left(\mathbb{N}\right)$
such that for each $f\in F$, 
\[
a+\sum_{t\in H}f\left(t\right)\in A.
\]
\item $A$ is said to be a CR-set if for each $k\in\mathbb{N}$ there exists
$r\in\mathbb{N}$ and whenever $F\in\mathcal{P}_{f}\left(S^{\mathbb{N}}\right)$
with $\mid F\mid\leq k$ , there exist $a\in S$ and $H\in\mathcal{P}_{f}\left(\mathbb{N}\right)$
with $\max H\leq r$ such that for each $f\in F$, 
\[
a+\sum_{t\in H}f\left(t\right)\in A.
\]
\item $A$ is said to be a k-CR-set if there exists $r\in\mathbb{N}$ and
whenever $F\in\mathcal{P}_{f}\left(S^{\mathbb{N}}\right)$ with $\mid F\mid\leq k$
, there exist $a\in S$ and $H\in\mathcal{P}_{f}\left(\mathbb{N}\right)$
with $\max H\leq r$ such that for each $f\in F$, 
\[
a+\sum_{t\in H}f\left(t\right)\in A.
\]
\end{enumerate}
One can easily observe that CR-sets are J-sets. 
\begin{defn}
Let $\left(S,+\right)$ be a commutative semigroup and let $A\subseteq S$. 
\begin{enumerate}
\item $J\left(S\right)=\left\{ p\in\beta S:\text{ for all }A\in p,A\text{ is a J-set}\right\} $.
\item $CR\left(S\right)=\left\{ p\in\beta S:\text{ for all }A\in p,A\text{ is a CR-set}\right\} $.
\item For $k\in\mathbb{N}$, $k-CR\left(S\right)=\left\{ p\in\beta S:\text{ for all }A\in p,A\text{ is a k-CR-set}\right\} $.
\end{enumerate}
\end{defn}

In the paper \cite[Def 4.3]{key-15} Patra stated the notion of J-sets
near zero. In the following, we write $\mathcal{T}_{0}$ for the set
of all sequences in $S$ which converge to zero under usual topology.
\begin{defn}
Let $\left(S,+\right)$ be a dense subsemigroup of $\left(\left(0,\infty\right),+\right)$
and $A\subseteq S$. Then 
\end{defn}

\begin{enumerate}
\item $A$ is said to be J-set near zero if and only if whenever $F\in P_{f}\left(\mathcal{T}_{0}\right)$
and $\delta>0$, there exist $a\in S\cap\left(0,\delta\right)$ and
$H\in P_{f}\left(\mathbb{N}\right)$ such that for each $f\in F$,
\[
a+\sum_{t\in H}f\left(t\right)\in A.
\]
\item $J^{0}\left(S\right)=\left\{ p\in0^{+}\left(S\right):\forall A\in p,A\text{ is J-set near zero}\right\} $.
\end{enumerate}
We define the CR-sets near zero and Essential CR-sets near zero in
the following.
\begin{defn}
Let $\left(S,+\right)$ be a dense subsemigroup of $\left(\left(0,\infty\right),+\right)$
and $A\subseteq S$.
\end{defn}

\begin{enumerate}
\item Then $A$ is a CR-set near zero if for each $k\in\mathbb{N}$ and
$\delta>0$ there exists $r\in\mathbb{N}$ whenever $F\in P_{f}\left(\mathcal{T}_{0}\right)$
with $\mid F\mid\leq k$ , there exist $a\in S\cap\left(0,\delta\right)$
and $H\subseteq\left\{ 1,2,\ldots,r\right\} $ such that for all $f\in F$,
\[
a+\sum_{t\in H}f\left(t\right)\in A.
\]
\item Then $A$ is a k-CR-set near zero if for each $\delta>0$ there exists
$r\in\mathbb{N}$ whenever $F\in P_{f}\left(\mathcal{T}_{0}\right)$
with $\mid F\mid\leq k$ , there exist $a\in S\cap\left(0,\delta\right)$
and $H\subseteq\left\{ 1,2,\ldots,r\right\} $ such that for all $f\in F$
\[
a+\sum_{t\in H}f\left(t\right)\in A.
\]
\end{enumerate}
It can easily be seen that CR-sets near zero are J-sets near zero.
In the following lemma, we show that the CR-sets near zero is partition
regular.
\begin{lem}
\label{PRCR near 0} Let $S$ be a dense subsemigroup of $\left(\left(0,\infty\right),+\right)$
and $A_{1},A_{2}\subseteq S$. If $A_{1}\cup A_{2}$ is a CR-set near
zero, then either $A_{1}$ or $A_{2}$ is a CR-set near zero.
\end{lem}

\begin{proof}
If possible let $A_{1}\cup A_{2}$ is a CR-set near zero in $S$ and
neither $A_{1}$ nor $A_{2}$ is a CR-set near zero in $S$.

For $i\in\left\{ 1,2\right\} $ pick $k_{i}\in\mathbb{N}$ and $\delta>0$
such that for every $r\in\mathbb{N}$ there exists $F\in\mathcal{P}_{f}\left(\mathcal{T}_{0}\right)$
such that $\mid F\mid\leq k_{i}$, for all $a\in S\cap\left(0,\delta\right)$
and for all $H\subseteq\left\{ 1,2,\ldots,r\right\} $, there exists
$f\in F$

\[
a+\sum_{t\in H}f(t)\notin A_{i}.
\]

Let $k=k_{1}+k_{2}$ and pick by \cite[Lemma 14.2.1]{key-13}, some
$n\in\mathbb{N}$ such that whenever the length $n$ words over $\left\{ 1,2,\ldots,k\right\} $
are 2-colored, there is a variable word $w\left(v\right)$ beginning
and ending with a constant and having no adjacent occurrences of $v$
such that $\left\{ w\left(l\right):l\in\left\{ 1,2,\ldots,k\right\} \right\} $
is monochromatic.

Let $h\in\mathcal{T}_{0}$. For each $t\in\mathbb{N}$, we can choose
$s_{t}\in\mathbb{N}$ such that $h\left(s_{t}\right)<\frac{1}{2^{t+1}}\delta$.
Since we could replace each $f_{i}$ by the sequence $f_{i}^{'}$,
defined by $f_{i}'\left(t\right)=f_{i}\left(s_{t}\right)$, we may
suppose that $\sum_{t=1}^{\infty}h\left(t\right)<\frac{1}{2}\delta$.

Let $W$ be the set of length $n$ words over $\left\{ 1,2,\ldots,k\right\} $.
For $w=b_{1}b_{2}\ldots b_{n}$ where each $b_{i}\in\left\{ 1,2,...,k\right\} $.
Define $g_{w}:\mathbb{N}\rightarrow S$ for $y\in\mathbb{N}$ by,
$g_{w}(y)=\sum_{i=1}^{n}f_{b_{i}}\left(ny+i\right)$. If $G=\left\{ g_{w}:w\in W\right\} $
, then $\mid G\mid\leq k^{n}=\alpha$ and $A_{1}\cup A_{2}$ is a
CR-set near zero in $S$, then for $\alpha\in\mathbb{N}$ there exists
$r_{1}\in\mathbb{N}$ and for $G\in\mathcal{P}_{f}\left(\mathcal{T}_{0}\right)$
such that $\mid G\mid\leq\alpha$ there exist $c\in\left(0,\frac{1}{2}\delta\right)$
and $r_{1}\in\mathbb{N}$ with $K\subseteq\left\{ 1,2,\ldots,r_{1}\right\} $
such that for all $w\in W$ 
\[
c+\sum_{y\in K}g_{w}\left(y\right)\in A_{1}\cup A_{2}.
\]
Define $\varphi:W\to\left\{ 1,2\right\} $ by $\varphi\left(w\right)=1$
if and only if $c+\sum_{y\in K}g_{w}\left(y\right)\in A_{1}$. Pick
a variable word $w\left(v\right)$, beginning and ending with a constant
and without successive occurrences of $v$ such that $\varphi$ is
constant on $\left\{ w\left(l\right):l\in\left\{ 1,2,\ldots,k\right\} \right\} .$
Assume without loss of generality that $\varphi\left(w\left(l\right)\right)=1$
for all $l\in\left\{ 1,2,\cdots,k\right\} ,$then for all $l\in\left\{ 1,2,\ldots,k\right\} $
\[
c+\sum_{t\in K}g_{w\left(l\right)}\left(t\right)\in A_{1}.
\]

Let $w\left(v\right)=b_{1}b_{2}\ldots b_{n}$ where each $b_{i}\in\left\{ 1,2,\ldots,k\right\} \cup\left\{ v\right\} $,
some $b_{i}=v,b_{1},b_{n}\neq v$, and if $b_{i}=v$, then $b_{i+1}\neq v$.

Let $C=\left\{ i\in\left\{ 1,2,\ldots,n\right\} :b_{i}\in\left\{ 1,2,\ldots,k\right\} \right\} $
and $I=\left\{ i\in\left\{ 1,2,\ldots,n\right\} :b_{i}=v\right\} $.
Then, for each $l\in\left\{ 1,2,\ldots,k\right\} $we have 

\[
c+\sum_{t\in K}g_{w\left(l\right)}\left(t\right)=c+\sum_{y\in K}\sum_{i\in C}f_{b_{i}}\left(ny+i\right)+\sum_{y\in K}\sum_{i\in I}f_{l}\left(ny+i\right).
\]

Let $a=c+\sum_{y\in K}\sum_{i\in C}f_{b_{i}}\left(ny+i\right)$.

Then $a\in S\cap\left(0,\delta\right)$ and $a+\sum_{y\in K}\sum_{i\in I}f_{l}\left(ny+i\right)\in A_{1}$
for every $l\in\left\{ 1,2,\ldots,k\right\} $.

Putting $H=nK+I$, then $H\subseteq\left\{ 1,2,\ldots,nr_{1}+n\right\} $,
we have 

\[
\sum_{y\in K}\sum_{i\in I}f_{l}\left(ny+i\right)=\sum_{y\in H}f_{l}\left(y\right)
\]

because the sets of the form $nK+i$, with $i\in I$, are pairwise
disjoint. It follows that $a+\sum_{t\in H}f_{l}(t)\in A_{1}$ for
every $l\in\left\{ 1,2,\ldots,k\right\} $ and hence a contradiction
arise.
\end{proof}
We know that 
\[
J\left(S\right)=\left\{ p\in\beta S:\forall A\in p,A\text{ is a J-set}\right\} 
\]
 is a two sided ideal of $\beta S$.

In \cite[Thorem 3.9]{key-4} Bayatmanesh, Tootkaboni proved that 
\[
J_{0}\left(S\right)=\left\{ p\in\beta S:\forall A\in p,A\text{ is a J-set near zero}\right\} 
\]
 is also a two sided ideal of $0^{+}\left(S\right)$.

Also 
\[
CR\left(S\right)=\left\{ p\in\beta S:\text{ for all }A\in p,A\text{ is a CR-set}\right\} 
\]

is a two sided ideal of $\beta S$ \cite[Theorem 2.6]{key-12}.

So we expect that $CR^{0}\left(S\right)=\left\{ p\in0^{+}\left(S\right):\forall A\in p,A\text{ is CR-set near zero}\right\} $
is also a two sided ideal of $0^{+}\left(S\right)$. The proof is
in the following.
\begin{lem}
\label{I} Let $S$ be a dense subsemigroup of $\left(\left(0,\infty\right),+\right)$.
Then $k\text{-}CR^{0}\left(S\right)$ is two sided ideal of $0^{+}\left(S\right)$
for all $k\in\mathbb{N}$.
\end{lem}

\begin{proof}
Let us first fixed $k$ and $p\in k\text{-}CR^{0}\left(S\right)$
and $q\in0^{+}\left(S\right)$. We have to show $p+q,q+p\in k\text{-}CR^{0}\left(S\right)$.

Let $A\in p+q$ , So $B=\left\{ x\in S:-x+A\in q\right\} \in p\in k\text{-}CR^{0}\left(S\right)$.
So by definition of $k\text{-}CR^{0}\left(S\right)$, $B$ is a k-CR-set
near zero. 

So there exists $r\in\mathbb{N}$ whenever $F\in P_{f}\left(\mathcal{T}_{0}\right)$
with $\mid F\mid<k$ and for any $\delta>0$, there exist $a\in S\cap\left(0,\frac{\delta}{2}\right)$,
and $H\subseteq\left\{ 1,2,\ldots,r\right\} $ such that for all $f\in F$
\[
a+\sum_{t\in H}f\left(t\right)\in B
\]

Then for all $f\in F$, 
\[
-\left(a+\sum_{t\in H}f\left(t\right)\right)+A\in q
\]
 i,e, 
\[
\bigcap_{f\in F}\left(-\left(a+\sum_{t\in H}f\left(t\right)\right)+A\right)\in q
\]

Let $y\in\left(\bigcap_{f\in F}\left(-\left(a+\sum_{t\in H}f\left(t\right)\right)+A\right)\right)\bigcap\left(0,\frac{\delta}{2}\right)$,
So $b+\sum_{t\in H}f\left(t\right)\in A$ for all $f\in F$, where
$b=y+a\in S\cap\left(0,\delta\right)$. So $A$ is a k-CR-set near
zero. Since $A$ is arbitrary element of $p+q$, Then $p+q\in k\text{-}CR^{0}\left(S\right)$.

Let $A\in q+p$ , Assumse $B=\left\{ x\in S:-x+A\in p\right\} \in q$
, $B$ is nonempty. For any $\delta>0$ there exists $x\in B\cap\left(0,\frac{\delta}{2}\right)$,
then $-x+A\in p\in k\text{-}CR^{0}\left(S\right)$. By definition
of $k\text{-}CR^{0}\left(S\right)$, $-x+A$ is a $k\text{-}CR$-set
near zero. So there exists $r\in\mathbb{N}$ whenever $F\in P_{f}\left(\mathcal{T}_{0}\right)$
with $\mid F\mid<k$ and for $\delta>0$, there exist $a\in S\cap\left(0,\frac{\delta}{2}\right)$,
and $H\subseteq\left\{ 1,2,\ldots,r\right\} $ such that for all $f\in F$
\[
a+\sum_{t\in H}f\left(t\right)\in-x+A.
\]

i.e, For all $f\in F$,
\[
b+\sum_{t\in H}f\left(t\right)\in A
\]

where $b=x+a\in S\cap\left(0,\delta\right)$ . So $A$ is a k-CR-set
near zero. Since $A$ is arbitrary element of $p+q$, Then $p+q\in k\text{-}CR^{0}\left(S\right)$.
So $k\text{-}CR^{0}\left(S\right)$ is two sided ideal of $0^{+}\left(S\right)$
for all $k\in\mathbb{N}$. And hence $CR^{0}\left(S\right)=\bigcup_{k\in\mathbb{N}}k\text{-}CR^{0}\left(S\right)$
is two sided ideal of $0^{+}\left(S\right)$.
\end{proof}
Then by Ellis theorem \cite[Corollary 2.39]{key-13}, from the lemma
(\ref{PRCR near 0}) and lemma (\ref{I}) we have $CR^{0}\left(S\right)$
is a compact right topological semigroup. Then $E\left(CR^{0}\left(S\right)\right)\neq\emptyset$,
in the next section we shall named the eliments from the idempotent
in $CR^{0}\left(S\right)$ are essential CR-sets near zero.

\section{Characterisation of essential CR-sets near zero}

There are several notions of largeness in topological dynamics, and
most of them play a significant role in Ramsey Theory. One such notion
is the central set. The concept of central set was first introduced
by H. Frustenberg using the notions of proximality and uniformly recurrent
from topological dynamics in \cite{key-10}. Later, Patra characterised
central sets near zero dynamically in \cite{key-15}. We present a
dynamic and basic characterization of essential CR sets near zero
in this section. Let's give a brief description of the well-known
large sets. We can begin by stating the well-known definition of a
topological dynamical system.
\begin{defn}
A dynamical system is a pair $\left(X,\left\langle T_{s}\right\rangle _{s\in S}\right)$
such that
\begin{enumerate}
\item $X$ is compact Hausdorff space,
\item $S$ is a semigroup,
\item For each $s\in S$, $T_{s}:X\to X$ and $T_{s}$ is continuous, and
\item For all $s,t\in S$, $T_{s}\circ T_{t}=T_{st}$.
\end{enumerate}
\end{defn}

We now state basic definitions, conventions and results for dynamical
characterization of members of certain idempotent ultrafilters in
the following.
\begin{defn}
Let $S$ be a nonempty discrete space and $\mathcal{K}$ be a filter
on $S$.
\begin{enumerate}
\item $\overline{\mathcal{K}}=\left\{ p\in\beta S:\mathcal{K}\subseteq p\right\} $
\item $L\left(\mathcal{K}\right)=\left\{ A\subseteq S:S\setminus A\notin\mathcal{K}\right\} $
\end{enumerate}
\end{defn}

\begin{thm}
\label{Kappa} Let $S$ be a nonempty discrete space and $\mathcal{K}$
be a filter on $S$.
\begin{enumerate}
\item $\overline{\mathcal{K}}=\left\{ p\in\beta S:A\in L\left(\mathcal{K}\right)\text{ for all }A\in p\right\} $
\item Let $\beta\subseteq L\left(\mathcal{K}\right)$ be closed under finite
intersections then there exists a $p\in\beta S$ with $\beta\subseteq p\subseteq L\left(\mathcal{K}\right)$.
\end{enumerate}
\end{thm}

\begin{proof}
Both of these assertions follows from \cite[Theorem 3.11]{key-13}.
\end{proof}
\begin{defn}
Let $S$ be a dense subsemigroup. Let $\left(X,\left\langle T_{s}\right\rangle _{s\in S}\right)$
be a dynamical system, $x$ and $y$ in $X$ and $\mathcal{K}$ be
a filter on $S$. The pair $\left(x,y\right)$ is called jointly $\mathcal{K}$-recurrence
if and only if for every neighborhood $U$ of $y$ we have 
\[
\left\{ s\in S:T_{s}\left(x\right),T_{s}\left(y\right)\in U\right\} \in L\left(\mathcal{K}\right).
\]
\end{defn}

\begin{thm}
\label{Th 1.8} Let $S$ be a nonempty discrete space and $\mathcal{K}$
be a filter on $S$ such that $\overline{\mathcal{K}}$ is a compact
subsemigroup of $\beta S$, and let $A\subseteq S$. Then $A$ is
a member of an idempotent in $\overline{\mathcal{K}}$ if and only
if there exists a dynamics system $\left(X,\left\langle T_{s}\right\rangle _{s\in S}\right)$
with points $x$ and $y$ in $X$ and there exists a neighborhood
$U$ of $y$ such that the pair $\left(x,y\right)$ is jointly $\mathcal{K}$-recurrence
and $A=\left\{ s\in S:T_{s}\left(x\right)\in U\right\} $.
\end{thm}

\begin{proof}
\cite[Theorem 3.3]{key-14}.
\end{proof}
\begin{lem}
\label{CR=00003DK} Let $S$ be a dense subsemigroup of $\left(\left(0,\infty\right),+\right)$
and 
\[
\mathcal{K}=\left\{ A\subseteq S:S\setminus A\text{ is not a CR-set near zero}\right\} .
\]
 Then $K$ is a filter on $S$ with $CR^{0}\left(S\right)=\overline{\mathcal{K}}$
and $CR^{0}\left(S\right)$ is a compact subsemigroup of $\beta S$.
\end{lem}

\begin{proof}
It's easy to verify $\mathcal{K}$ is nonempty. If $A\in\mathcal{K}$
and $A\subseteq B$, then $B\in\mathcal{K}$. From lemma (\ref{PRCR near 0}),
$\mathcal{K}$ is closed under finite intersection, so $\mathcal{K}$
is a filter. Now we have $L\left(\mathcal{K}\right)=\left\{ A\subseteq S:A\text{ is a CR-set near zero}\right\} $.
From theorem (\ref{Kappa}) , $CR^{0}\left(S\right)=\overline{\mathcal{K}}$.
By Ellis theorem \cite[Corollary 2.39]{key-13}, from the lemma (\ref{PRCR near 0})
and lemma (\ref{I}) we have $CR^{0}\left(S\right)$ is a compact
subsemigroup of $\beta S$.
\end{proof}
Here we define essential CR-set near zero in the following.
\begin{defn}
\label{def:ECR} Let $S$ be a dense subsemigroup of $\left(\left(0,\infty\right),+\right)$
and $A\subseteq S$. Then $A$ is an essential CR-set near zero if
there exists an idempotent $p\in CR^{0}\left(S\right)$ such that
$A\in p$.
\end{defn}

The following theorem is the dynamical characterization of essential
CR-set near zero.
\begin{thm}
\label{DC CR near 0} Let $S$ be a dense subsemigroup of $\left(\left(0,\infty\right),+\right)$
and $A\subseteq S$. Then $A$ is an essential CR-set near zero if
and only if there exists a dynamical system $\left(X,\left\langle T_{s}\right\rangle _{s\in S}\right)$
with points $x$ and $y$ in $X$ such that for every neighbourhood
$U$ of $y$, $\left\{ s\in S:T_{s}\left(x\right),T_{s}\left(y\right)\in U\right\} $
is a CR-set near zero and a neighbourhood $V$ of $y$ such that $A=\left\{ s\in S:T_{s}\left(x\right)\in V\right\} $.
\end{thm}

\begin{proof}
Let $\mathcal{K}=\left\{ B\subseteq S:S\setminus B\text{ is not a CR-set near zero}\right\} $.
By lemma (\ref{CR=00003DK}),

$CR^{0}\left(S\right)=\overline{\mathcal{K}}$ and $L\left(\mathcal{K}\right)=\left\{ A\subseteq S:A\text{ is a CR-set near zero}\right\} $
then we apply theorem (\ref{Th 1.8}), we get our desire result.
\end{proof}
In \cite[Theorem 2.3]{key-7} De, Debnath and Goswami deduced the
combinatorial characterization of essential $\mathcal{F}$-sets. Here
we like to characterise the essential CR-set near zero combinatorially.
For this we need to recall the following definitions.

Let $\omega$ be the first infinite ordinal and each ordinal indicates
the set of all it\textquoteright s predecessor. In particular, $0=\emptyset$,
for each $n\in\mathbb{N}$, $n=\left\{ 0,1,\ldots,n-1\right\} $.
\begin{defn}
\cite[Definition 2.5]{key-13-1} 1. If $f$ is a function and $dom\left(f\right)=n\in\omega$,
then for all $x$, $f^{\frown}x=f\cup\left\{ \left(n,x\right)\right\} $.

2. Let $T$ be a set functions whose domains are members of $\omega$.
For each $f\in T$ , $B_{f}\left(T\right)=\left\{ x:f^{\frown}x\in T\right\} $.
\end{defn}

The following is the key lemma for our characterization.
\begin{lem}
\label{For Eli. Cha.}\cite[Lemma 2.6]{key-13-1} Let $p\in\beta S$.
Then $p$ is an idempotent if and only if for each $A\in p$ there
is a non-empty set $T$ of functions such that 
\begin{enumerate}
\item For all $f\in T$, $dom\left(f\right)\in\omega$, and $range\left(f\right)\subseteq A$.
\item For all $f\in T$, $B_{f}\left(T\right)\in p$.
\item For all $f\in T$ and any $x\in B_{f}\left(T\right)$, $B_{f^{\frown}x}\left(T\right)\subseteq x^{-1}B_{f}\left(T\right)$.
\end{enumerate}
\end{lem}

The following theorem is the characterization of essential CR-sets
near zero.
\begin{thm}
\label{Elimentary characterization} Let $S$ be a dense subsemigroup
of $\left(\left(0,\infty\right),+\right)$and $A\subseteq S$. Then
the followings are equivalent.
\begin{enumerate}
\item $A$ is an essential CR-set near zero.
\item There is a non empty set $T$ of functions such that
\begin{enumerate}
\item For all $f\in T$, $dom\left(f\right)\in\omega$ and $rang\left(f\right)\subseteq A$. 
\item For all $f\in T$ and all $x\in B_{f}\left(T\right)$, $B_{f^{\frown}x}\subseteq x^{-1}B_{f}$.
\item For each $F\in P_{f}\left(T\right)$, $\bigcap_{f\in F}B_{f}\left(T\right)$
is a CR-set near zero.
\end{enumerate}
\item There is a downward directed family $\left\langle C_{F}\right\rangle _{F\in I}$
of subsets of $A$ such that
\begin{enumerate}
\item For each $F\in I$ and each $x\in C_{F}$ there exists $G\in I$ with
$C_{G}\subseteq x^{-1}C_{F}$ and
\item For all $\mathcal{F}\in P_{f}\left(I\right)$, $\bigcap_{F\in\mathcal{F}}C_{F}$
is a CR-set near zero.
\end{enumerate}
\end{enumerate}
\end{thm}

\begin{proof}
(1) \ensuremath{\Rightarrow} (2) As $A$ be an essential CR-set near
zero, then there exists an idempotent $p\in CR^{0}\left(S\right)$
such that $A\in p$. Pick a set $T$ of functions as guaranteed by
Lemma (\ref{For Eli. Cha.}). Conclusions (a) and (b) hold directly.
Given $F\in P_{f}\left(T\right)$, $B_{f}\in p$ for all $f\in F$,
hence $\bigcap_{f\in F}B_{f}\left(T\right)\in p$ and so $\bigcap_{f\in F}B_{f}\left(T\right)$
is a CR-set near zero.

(2) \ensuremath{\Rightarrow} (3) Let $T$ be guaranteed by (2). Let
$I=P_{f}\left(T\right)$ For each $F\in I$, let $C_{F}=\bigcap_{f\in F}B_{f}\left(T\right)$
. Then directly each $C_{F}$ is a CR-set near zero. Given $\mathcal{F}\in P_{f}\left(I\right)$,
if $G=\cup\mathcal{F}$, then $\bigcap_{F\in\mathcal{F}}C_{F}=C_{G}$
and is therefore a CR-set near zero. To verify (a), let $F\in I$
and let $x\in C_{F}$ . Let $G=\left\{ f^{\frown}x:f\in F\right\} $.
For each $f\in F$, $B_{f^{\frown}x}\subseteq x^{-1}B_{f}$ and so
$C_{G}\subseteq x^{-1}C_{F}$.

(3) \ensuremath{\Rightarrow} (1) Let $\left\langle C_{F}\right\rangle _{F\in I}$
is guaranteed by (3). Let $M=\bigcap_{F\in I}\overline{C_{F}}$ .
By {[}12, Theorem 4.20{]}, $M$ is a subsemigroup of $\beta S$. By
{[}12, Theorem 3.11{]} there is some $p\in\beta S$ such that $\left\{ C_{F}:F\in I\right\} \subseteq p\subseteq\mathcal{F}$.
Therefore $M\cap CR^{0}\left(S\right)\neq\emptyset$; and so $M\cap CR^{0}\left(S\right)$
is a compact subsemigroup of $\beta S$. Thus there is an idempotent
$p\in M\cap CR^{0}\left(S\right)$, and so $A$ is an essential CR-set
near zero.
\end{proof}

\section{Cartesian product of essential cr-sets}

The cartesian product of two central sets is a central set; the cartesian
product of two J-sets is a J-set; and the cartesian product of two
C-sets is a C-set. In \cite{key-5} Debnath demonstrated that the
product of two C-sets is a C-set by examining the dynamical characterization
of C-sets. Recently, Goswami has demonstrated that the product of
two CR-sets is also a CR-set in \cite{key-11}. We are looking into
whether the same result is true for CR-sets near zero and essential
CR-sets near zero.

We are now denote some useful notations in the following.

Let $S$ and $T$ be dense subsemigroups of $\left(\left(0,\infty\right),+\right)$.
Then 
\[
\mathcal{D}_{0}=\left\{ \begin{array}{c}
\left(z_{n}=\left(x_{n},y_{n}\right)\right)_{n\in\mathbb{N}}\in\left(S\times T\right)^{\mathbb{N}}:\\
\text{ both }x_{n},y_{n}\text{ converges to zero under usual topology}
\end{array}\right\} ,
\]

\[
\mathcal{T}_{0}=\left\{ \left(x_{n}\right)_{n\in\mathbb{N}}\in S^{\mathbb{N}}:x_{n}\text{ converges to zero under usual topology}\right\} 
\]
 and 
\[
\mathcal{U}_{0}=\left\{ \left(y_{n}\right)_{n\in\mathbb{N}}\in T^{\mathbb{N}}:y_{n}\text{ converges to zero under usual topology}\right\} .
\]

In the following definition, we introduce the CR-sets near zero in
the cartetian product of two semigroups which are dense in $\left(\left(0,\infty\right),+\right)$.
\begin{defn}
Let $S$ and $T$ be a dense subsemigroup of $\left(\left(0,\infty\right),+\right)$
and $D\subseteq S\times T$ is called CR-set near zero in $S\times T$
if for $k\in\mathbb{N}$, $\delta>0$ there exists $r\in\mathbb{N}$,
for every $F\in\mathcal{P}_{f}\left(\mathcal{D}_{0}\right)$ with
$\mid F\mid\leq k$ there exist $\overline{a}=\left(a_{1},a_{2}\right)\in\left(S\times T\right)\bigcap\left(0,\delta\right)^{2}$
and $H\subseteq\left\{ 1,2,\ldots,r\right\} $ such that for all $f\in F$
\[
\overline{a}+\sum_{t\in H}f\left(t\right)\in D.
\]
\end{defn}

A set $A$ in $P_{f}\left(\mathbb{N}\right)$ is said to be an IP-set
if there is a sequence $\langle H_{n}\rangle_{n}$ in $P_{f}\left(\mathbb{N}\right)$
such that $\max H_{n}<\min H_{n+1}$ for each $n\in\mathbb{N}$ and
$A=FU\left(\langle H_{n}\rangle_{n}\right)=\left\{ \bigcup_{n\in K}H_{n}:K\in P_{f}\left(\mathbb{N}\right)\right\} $.
For every $r\in\mathbb{N}$, a set $B$ is called an $IP_{r}$ set
if there exists sequence $\langle H_{n}\rangle_{n=1}^{r}$ in $P_{f}\left(\mathbb{N}\right)$
such that $\max H_{n}<\min H_{n+1}$ for each $n\in\mathbb{N}$ and
\[
B=FU\left(\langle H_{n}\rangle_{n=1}^{r}\right)=\left\{ \bigcup_{n\in K}H_{n}:K\left(\neq\emptyset\right)\subseteq\left\{ 1,2,\ldots,r\right\} \right\} .
\]
 A set is called $IP^{*}$ (resp. $IP_{r}^{*}$ ) if that set intersects
with every $IP$ sets (resp. $IP_{r}$ sets).
\begin{lem}
\label{intersection} Let $S$ be a dense subsemigroups of $\left(\left(0,\infty\right),+\right)$
and $k\in\mathbb{N}$. Let $A$ be a CR-set near zero in $S$, let
$F\in P_{f}\left(\mathcal{T}_{0}\right)$ with $|F|\leq k$ and $\delta>0$.
Let $\Theta=\left\{ H\in P_{f}\left(\mathbb{N}\right):\left(\exists a\in S\cap\left(0,\delta\right)\right)\left(\forall f\in F\right)\left(a+\sum_{t\in H}f\left(t\right)\in A\right)\right\} .$
Let $r=r\left(A,k,\delta\right)$, and $\langle H_{n}\rangle_{n=1}^{r}$
be a sequence in $P_{f}\left(\mathbb{N}\right)$ such that $\max H_{n}<\min H_{n+1}$
for each $n\in\left\{ 1,2,\ldots,r-1\right\} $. Then there exists
$K\subseteq\left\{ 1,2,\ldots,r\right\} $ such that $\bigcup_{n\in K}H_{n}\in\Theta$.
In other words, for each $k\in\mathbb{N},$ $\Theta$ is an $IP_{r}^{*}$
set, where $r=r\left(A,k,\delta\right)$.
\end{lem}

\begin{proof}
For each $n\in\left\{ 1,2,\ldots,r\right\} $, let $\alpha_{n}=|H_{n}|$
and write 
\[
H_{n}=\left\{ 1,2,\ldots,\alpha_{n}\right\} .
\]
For $f\in F$, define $g_{f}\in\mathcal{T}_{0}$ by,

\[
g_{f}\left(n\right)=\sum_{t\in H_{n}}f\left(n\right)\text{ if }n\in\left\{ 1,2,\ldots,r\right\} ,
\]
 and $g_{f}(n)=f\left(n\right)$ otherwise. Now $\{g_{f}:f\in F\}\in P_{f}\left(\mathcal{T}_{0}\right)$
with $|\{g_{f}:f\in F\}|\leq k.$ So pick $a\in S\cap\left(0,\delta\right)$
and $K\subseteq\left\{ 1,2,\ldots,r\right\} $ such that $a+\sum_{t\in K}g_{f}\left(t\right)=a+\sum_{t\in K}\sum_{n\in H_{t}}f\left(n\right)\in A$.
Therefore $\bigcup_{n\in K}H_{n}\in\Theta$.
\end{proof}
To prove the desired theorem we also need the following lemma.
\begin{lem}
\label{IP} Let $r,s\in\mathbb{N}$. Let $A$ and $B$ be $IP_{r}^{*}$
set (resp. $IP_{s}^{*}$ set) in $P_{f}\left(\mathbb{N}\right)$.
Then there exists $l=l\left(r,s\right)\in\mathbb{N}$ such that $A\cap B$
is an $IP_{l}^{*}$ set.
\end{lem}

\begin{proof}
\cite[Proposition 2.5]{key-3}.
\end{proof}
\begin{thm}
\label{Product} Let $S$ and $T$ be a dense subsemigroups of $\left(\left(0,\infty\right),+\right)$.
Let $A$ be n CR-set near zero in $S$ and $B$ be a CR-set near zero
in $T$. Then $A\times B$ is CR-set near zero in $S\times T$.
\end{thm}

\begin{proof}
Let $A\subseteq S$ and $B\subseteq T$ be two CR-set near zero. Let
$k\in\mathbb{N}$ and $F\in\mathcal{P}_{f}\left(\mathcal{D}_{0}\right)$
with $|F|\leq k.$ Let $G=\{\pi_{1}\circ f:f\in F\}$ and $H=\{\pi_{2}\circ f:f\in F\}$.
Then $G\in\mathcal{P}_{f}\left(\mathcal{T}_{0}\right)$ and $H\in\mathcal{P}_{f}\left(\mathcal{U}_{0}\right)$
with $\mid G\mid=\mid H\mid\leq k.$ Now from Lemma \ref{intersection},
For every $\delta>0$ there exist $r,s\in\mathbb{N}$ such that 
\[
\Theta_{1}=\left\{ H\in P_{f}\left(\mathbb{N}\right):\left(\exists a\in S\cap\left(0,\delta\right)\right)\left(\forall f\in G\right)\left(a+\sum_{t\in H}f\left(t\right)\in A\right)\right\} 
\]
 and 
\[
\Theta_{2}=\left\{ H\in P_{f}\left(\mathbb{N}\right):\left(\exists b\in S\cap\left(0,\delta\right)\right)\left(\forall f\in H\right)\left(b+\sum_{t\in H}f\left(t\right)\in B\right)\right\} 
\]
 are $IP_{r}^{*}$ and $IP_{s}^{*}$ sets respectively. So $\Theta=\Theta_{1}\cap\Theta_{2}$
is an $IP_{l}^{*}$ set, where $l=l(r,s)$ is coming from Lemma \ref{IP}.
consequently $\Theta_{1}\cap\Theta_{2}\cap FU\left(\left\langle \left\{ i\right\} \right\rangle _{i=1}^{l}\right)\neq\emptyset$,
pick $H\in\Theta_{1}\cap\Theta_{2}\cap FU\left(\left\langle \left\{ i\right\} \right\rangle _{i=1}^{l}\right).$

Hence there exist $H\subseteq\left\{ 1,2,\ldots,l\right\} $, $\left(a,b\right)\in\left(S\times T\right)\cap\left(0,\delta\right)^{2}$
such that for all $f\in F$, 
\[
\left(a,b\right)+\sum_{t\in H}f\left(t\right)\in A\times B.
\]

This completes the proof.
\end{proof}
\begin{thm}
Let $S$ and $T$ be a dense subsemigroups of $\left(\left(0,\infty\right),+\right)$.
Let $A$ and $B$ be an essential CR-set near zero in $S$ and $T$
respectively. Then $A\times B$ is essential CR-set near zero in $S\times T$.
\end{thm}

\begin{proof}
Using the elimentary characterization of essential CR-set near zero
(Theorem \ref{Elimentary characterization}) and the above result,
its easy to verify.
\end{proof}

\end{document}